\documentclass{article}
\usepackage[english]{babel}
\usepackage[utf8]{inputenc}
\usepackage[T1]{fontenc}
\usepackage{graphicx}
\usepackage{geometry}
\usepackage{sectsty}
\usepackage{dsfont}
\usepackage{alltt}
\usepackage{amsthm,amsfonts}
\usepackage{amsmath}
\usepackage{mathabx}
\usepackage{titlesec}
\usepackage{empheq}
\mathtoolsset{showonlyrefs=true}

\newcommand{\ce}{\mathbb{C}}
\newcommand{\be}{\mathbb{B}}
\newcommand{\er}{\mathbb{R}}
\newcommand{\R}{\mathbb{R}}
\newcommand{\en}{\mathbb{N}}
\newcommand{\qe}{\mathbb{Q}}
\newcommand{\e}{\epsilon}
\newcommand{\vertiii}[1]{{\left\vert\kern-0.25ex\left\vert\kern-0.25ex\left\vert #1 
		\right\vert\kern-0.25ex\right\vert\kern-0.25ex\right\vert}}

\newcommand{\mf}[1]{\mathbf{#1}}
\newcommand{\ub}{\mathbf{u}}

\newtheorem{Teorema}{Theorem}[section]
\newtheorem{Propriedade}[Teorema]{Proposition}
\newtheorem{Rem}[Teorema]{Remark}
\newtheorem{Prop}[Teorema]{Proposition}
\newtheorem{Lemma}[Teorema]{Lemma}
\newtheorem{Coro}[Teorema]{Corollary}

\DeclareMathOperator{\sech}{sech}
\newcommand\blfootnote[1]{%
  \begingroup
  \renewcommand\thefootnote{}\footnote{#1}%
  \addtocounter{footnote}{-1}%
  \endgroup
}
\date{}
\author{Sim\~ao Correia and Filipe Oliveira}

\title{Scattering theory for the Schr\"odinger-Debye System}
\begin{document}
	\maketitle
	\begin{abstract}
	We study the Schr\"odinger-Debye system over $\er^d$
	\begin{equation*}
	\left\{\begin{array}{llll}
	iu_t+\frac 12\Delta u=uv,\\
	\mu v_t+v=\lambda |u|^2
	\end{array}\right.
	\end{equation*}
	and establish the global existence and scattering of small solutions for initial data in several function spaces in dimensions $d=2,3,4$. Moreover, in dimension $d=1$, we prove a Hayashi-Naumkin modified scattering result.\\
	
	\smallskip
	
	\end{abstract}

\section{Introduction}
In this work, we consider the initial value problem associated to the Schr\"odinger-Debye system
\begin{equation}
\label{SD}
\left\{\begin{array}{llll}
iu_t+\frac 12\Delta u=uv\qquad (x,t)\in\er^d\times \er,\ d\le 4,\\
\mu v_t+v=\lambda |u|^2\qquad \mu>0, \lambda=\pm 1\\
u(\cdot,0)=u_0,\,\qquad\, v(\cdot,0)=v_0.
\end{array}\right.,
\end{equation}
where the function $u$ is complex-valued and $v$ is real. From a physical point of view, this model may describe the dynamics
of an electromagnetic wave propagating through a nonresonant medium whose response time cannot be considered instantaneous (see Newell and Mahoney \cite{newell} and the references therein). When $\mu\to 0$, this response time vanishes and, in the limit case $\mu=0$, we obtain the classical cubic nonlinear Schr\"odinger equation
\begin{equation}\tag{NLS}
iu_t + \frac 12 \Delta u = \lambda |u|^2u.
\end{equation}
In this framework, one says that (NLS) is focusing (resp. defocusing) when $\lambda=-1$ (resp. $\lambda=1$).

The Cauchy problem \eqref{SD} has been studied by several authors. Noticing that the second equation in \eqref{SD} can be integrated in time, yielding
\begin{equation}
\label{vintegrada}
v(t)=e^{-t/\mu}v_0+\frac{\lambda}{\mu}\int_0^te^{-(s-t)/\mu}|u(s)|^2ds,
\end{equation}
one formally obtains
$$
iu_t + \frac 12 \Delta u = \lambda u \left(e^{-t/\mu}v_0+\frac{\lambda}{\mu}\int_0^te^{-(s-t)/\mu}|u(s)|^2ds\right),
$$
or, in integral form,
\begin{equation}
\label{duhamel}
u(t)=S(t)u_0+i\int_0^t S(t-s)\Big(e^{-s/\mu}v_0+\frac{\lambda}{\mu}\int_0^se^{-(s-s')/\mu}|u(s')|^2ds'\Big)u(s)ds,
\end{equation}
where $(S(t))_{t\in\er}=(e^{\frac{i}{2}\Delta t})_{t\in\er}$ is the unitary Schr\"odinger group. We say that $(u,v)$ is a (weak) solution of the Cauchy problem \eqref{SD} on if $(u,v)$ satisfies (\ref{duhamel}-\ref{vintegrada}). For convenience, we synthethize several results shown in \cite{bid1}, \cite{bid2}, \cite{linares}, \cite{corchosilva}, \cite{corchooliveira} in the following two theorems:

\begin{Teorema}
	\label{bi}
	Let $(u_0,v_0)\in H^s(\er^d)\times H^s(\er^d)$, where $d=1,2,3$ and $s=0,1$.
	Then there exists a unique solution to the Cauchy Problem (\ref{SD}), with
	$$(u,v)\in C([0,+\infty); H^s(\er^d)\times H^s(\er^d)).$$ Furthermore, if
	$$
	u_0\in \Sigma(\er^d):= \{f\in H^1(\er^d): |x|f \in L^2(\er^d)\},
	$$
	then $u\in C([0,\infty), \Sigma(\er^d))$.
\end{Teorema}

\begin{Teorema}
	\label{lwp4} Let $(u_0,v_0)\in H^1(\er^4)\times H^1(\er^4)$. Then, there exists $T=T(\|u_0\|_{H^1},\|v_0\|_{H^1})$ and a unique solution to the initial value problem (\ref{SD}) in the time interval $[0,T]$ satisfying 
	$$(u,v)\in C([0,T];H^1(\er^4)\times H^1(\er^4)),$$
	with 
	$$\|u\|_{L^{\infty}((0,T), H^1_x)}+\| u\|_{L^2((0,T), W^{1,4}_x)}+\|v\|_{L^\infty((0,T),H^1_x)}<+\infty.$$
	Moreover, if $u_0\in \Sigma(\er^4)$, then $u\in C([0,T], \Sigma(\er^4))$.
\end{Teorema}

With respect to local existence results, we improve the existing theory in $d=4$ by showing
\begin{Teorema}
	\label{lwp_L2_d=4} Let  $(u_0,v_0)\in L^2(\er^4)\times L^2(\er^4)$, $\|v_0\|_{L^2}$ sufficiently small. Then, there exists a unique maximal solution to the initial value problem (\ref{SD}) $(u,v)$ defined on $[0,T)$ such that
	$$
	u\in C([0,t], L^2(\er^4)) \cap L^2((0,t), L^4(\er^4)),  v\in C([0,t]. L^2(\er^4)), \quad t<T.
	$$
	If $T<\infty$, then
	$$
	\lim_{t\to T} \|u\|_{L^2((0,t), L^4(\er^4))} = \infty.
	$$
	Furthermore, there exists $\epsilon>0$ such that, if $\|u_0\|_{L^2}<\epsilon$, then $T=\infty$ and 
	$$
	\|u\|_{L^\infty((0,\infty), L^2(\er^4)} + \|v\|_{L^\infty((0,\infty), L^2(\er^4)} + \|u\|_{L^2((0,\infty), L^4(\er^4))} <\infty.
	$$
\end{Teorema}

The system \eqref{SD} exhibits a pseudo-scaling invariance: if	$(u,v)$ is a solution of \eqref{SD}, then 
\begin{equation}
\label{scaling}
(u_{\mu^*},v_{\mu^*}):=\Big(\Big(\frac{\mu^*}{\mu}\Big)^{\frac 12}u\Big(\Big(\frac{\mu^*}{\mu}\Big)^{\frac 12}x\,,\,\frac{\mu^*}{\mu} t\Big),\frac{\mu^*}{\mu}v\Big(\Big(\frac{\mu^*}{\mu}\Big)^{\frac 12}x,\frac{\mu^*}{\mu} t\Big)\Big)
\end{equation}
is the solution of \eqref{SD} with $\mu$ replaced by $\mu^*$. Consequently, for convenience of notations, we will fix $\mu=1$ for the remainder of this work, even though our results remain valid for $\mu>0$.

Notice that (NLS) is invariant by \eqref{scaling}. Considering that the Schr\"odinger-Debye system is a perturbation of (NLS), some authors use the definition of criticallity for the latter in the context of system \eqref{SD}: since
$$
\|u_{\mu^*}(t)\|_{L^2(\er^d)} =\left(\frac{\mu^*}{\mu}\right)^{\frac{2-d}{4}} \|u_{\mu}(t)\|_{L^2(\er^d)}, \|\nabla u_{\mu^*}(t)\|_{L^2(\er^d)} =\left(\frac{\mu^*}{\mu}\right)^{\frac{4-d}{4}} \|\nabla u_{\mu}(t)\|_{L^2(\er^d)},
$$
one says that dimension $d=2$ is $L^2$-critical and dimension $d=4$ is $H^1$-critical. Note, however, that the local well-posedness results stated in Theorems \ref{bi} and \ref{lwp4} suggest that this analogy may be accurate. Furthermore, the conditional local well-posedness result stated in our Theorem \ref{lwp_L2_d=4} seems to indicate that it would be more adequate to consider the dimension $d=4$ as being $L^2$-critical. 
\medskip

The large time behaviour of small solutions is a central research topic in the field of nonlinear dispersive equations (see for instance \cite{137}, \cite{138}, \cite{133}, \cite{341}, \cite{Cazenave-Weissler}). Generally speaking, one expects small solutions to be globally defined and to \textit{scatter}, that is, to tend, in a certain sense, to a solution of the corresponding linear problem. To the best of our knowledge, no such results regarding the Schr\"odinger-Debye system are currently available in the literature. In the present work, we will show several scattering results for small solutions of \eqref{SD}:

\begin{Teorema}[Scattering in dimension $d=4$]\label{teo:scat4}
Let $$(X,Y)\in\{(L^2(\er^4),L^2(\er^4)), (H^1(\er^4), H^1(\er^4)), (\Sigma(\er^4), H^1(\er^4))\}.$$
There exists $\epsilon>0$ such that, if $(u_0,v_0)\in X\times Y$ satisfies $\|u_0\|_{X}+\|v_0\|_{Y}<\epsilon$, then the corresponding solution $(u,v)$ of (\ref{SD}) is global and scatters, that is, there exists $u_+\in X$ such that
\begin{equation}\label{eq:scattering}
\|u(t)-S(t)u_+\|_{X} \to 0 \mbox{ and }  \|v(t)\|_{Y}\to 0, \quad t\to \infty.
\end{equation}
In the particular case $(X,Y)=(\Sigma(\er^4), H^1(\er^4))$, the following decay estimate holds:
\begin{equation}\label{co2}
\|u(t)\|_{L^p}\lesssim \frac{C(\|u_0\|_\Sigma, \|v_0\|_{H^1})}{t^{\left(2-\frac{4}{p}\right)}}, \quad t>0,\ 2<p< 4.
\end{equation}
\end{Teorema}

\begin{Teorema}[Scattering in dimensions $d=2,3$]\label{scat23}
	There exists $\delta>0$ such that, if $(u_0,v_0)\in \Sigma(\er^d)\times H^1(\er^d)$, $d=2,3$, satisfies $\|u_0\|_{H^1}+\|v_0\|_{H^1}<\delta$, then the corresponding solution $(u,v)$ of (\ref{SD}) scatters, that is, there exists $u_+\in \Sigma(\er^d)$ such that
	$$
	\|u(t)-S(t)u_+\|_{\Sigma} \to 0, \ \|v(t)\|_{H^1}\to 0, \quad t\to \infty.
	$$
	Furthermore,
	\begin{equation}\label{co3}
	\|u(t)\|_{L^p}\lesssim \frac{C(\|u_0\|_\Sigma, \|v_0\|_{H^1})}{t^{d\left(\frac{1}{2}-\frac{1}{p}\right)}}, \quad t>0,\ 2<p< 2d/(d-2)^+.
	\end{equation}
\end{Teorema}

In what concerns the dimension $d=1$, we observe that the Schr\"odinger Debye system exhibits a critical decay behaviour: indeed, if $u(t)=S(t)u_0$ and $v$ is defined by \eqref{vintegrada}, then the linear decay of the unitary group implies, for large times, the critical decay of the nonlinear potential $v$:
$$
\|v(t)\|_{L^\infty(\er)} \sim \frac{1}{t}.
$$
As in the (NLS) case, this seriously compromises the chances of proving a classical result in dimension $d=1$. In this scenario, scattering results may be obtained provided that a phase correction in the Fourier space is introduced, the so-called \textit{modified scattering}, \textit{i.e.}, denoting the Fourier transform by $\ \widehat{\cdot}$ ,
$$
\| e^{i\Psi(t)}\widehat{S(-t)u(t)} - \hat{u_+} \|_X \to 0, \quad t\to\infty,
$$
for some real-valued function $\Psi$ and some Banach space $X$ (see \cite{181}, \cite{184}). With this in mind, following the ideas in \cite{Pusateri}, we show

\begin{Teorema}\label{modified}
	There exists $\epsilon>0$ such that, if $(u_0,v_0)\in \Sigma(\er)\times H^1(\er)$ satisfies $\|u_0\|_{H^1}+\|v_0\|_{H^1}<\epsilon$, then the corresponding solution $(u,v)$ of (\ref{SD}) scatters up to a phase correction, that is, there exists (a unique) $u_+\in L^{2}(\er)$ such that
	$$
	\|e^{i\Psi(t)}\widehat{S(-t)u(t)} - \hat{u_+}\|_{L^{2}} \to 0, \ \|v(t)\|_{L^{\infty}}\to 0, \quad t\to \infty,
	$$
	where $\displaystyle\Psi(\xi,t)=\int_1^t\int_1^s\frac 1{2s'}e^{-(s-s')}\Big|\hat{f}\Big(\frac s{s'}\xi,s'\Big)\Big|^2ds'ds$.
	Also, $$\|u(t)\|_{L^{\infty}}\lesssim \frac 1{t^{\frac 12}},\quad t\to+\infty.$$
\end{Teorema} 

\begin{Rem}
As in \cite{[181]}, the proof of Theorem \ref{modified} will imply $\hat{u}_+\in L^\infty(\er)$ and also an asymptotic expansion for the phase correction: for some $\Phi\in L^\infty(\er)$, one has
$$
\left\| e^{i\Psi(t)} - |\hat{u}_+|^2\log t - \Phi \right\|_\infty \to 0,\quad t\to \infty.
$$
\end{Rem}

\noindent\textbf{Notations.} We denote the Fourier transform in the spatial variable by $\widehat{\cdot}$. If the Fourier transform is taken with respect to a particular variable $\xi$, we instead write $\mathcal{F}_\xi$. The spatial domain $\er^d$ will often be ommited. We abbreviate $L^q((0,T), L^p(\er^d))$ as $L^q_TL^r_x$. The Sobolev space of elements in $L^p(\er^d)$ with derivatives in $L^p(\er^d)$ up to order $k$ is denoted by $W^{k,p}(\er^d)$. In the particular case where $p=2$, we set $H^s(\er^d):=W^{s,2}(\er^d)$. We use the convention $2/(d-2)^+=2/(d-2)$, for $d\ge 3$, $2/(d-2)^+=+\infty$, if $d=1,2$. 
\medskip

We end this introduction by recalling some Strichartz estimates for the unitary Schr\"odinger group, which we will use throughout the paper: we say that $(q,r)$ is an admissible pair in dimension $d$ if
$$
2\le r \le 2d/(d-2)^+, \quad \frac{2}{q}=d\left(\frac{1}{2}-\frac{1}{r}\right),\quad  (q,r)\neq (2,\infty).
$$
Then, for any couple of admissible pairs $(q,r)$ and $(\gamma,\rho)$, the following estimates hold:
$$
\| S(t)u\|_{L^q_tL^r_x} \lesssim \|u\|_{L^2}, \quad t\in\er, \quad \mbox{(Homogeneous estimate)}
$$
$$
\left\| \int_0^t S(t - s) f(s) ds \right\|_{L^q_tL^r_x}\lesssim \|f\|_{L^{\gamma'}_tL^{\rho'}_x} \quad \mbox{(Inhomogeneous estimate)}.
$$

%
%
%
%
\section{Scattering theory in dimension $d=4$}

\noindent
In this section, we will establish scattering, global existence and decay of the solutions of the Schr\"odinger-Debye system \eqref{SD} for small initial data in several spaces. 

\subsection{$L^2(\er^4)\times L^2(\er^4)$ theory}

First, we settle the local well-posedeness result for initial data in $L^2$.

\begin{proof}[Proof of Theorem \ref{lwp_L2_d=4}]
The proof of this result can be obtained by standard arguments (s
ee \cite{kato1}) from the \textit{a priori} estimates
\begin{equation}
\label{lwpL2L4}
\|u\|_{L_t^2L_x^{4}} \lesssim \|S(t)u_0\|_{L^2_t L^4_x}+\|u\|_{L^2_tL^4_x}^3,
\end{equation}
\begin{equation}
\label{lwpLinftyL2}
\|u\|_{L_t^\infty L_x^{2}} \lesssim \|u_0\|_{L^2}+\|u\|_{L^2_tL^4_x}^3,
\end{equation}
and
$$\|v\|_{L^{\infty}_{t}L^2_x}\leq \|v_0\|_{L^2}+\|u\|_{L^2_{t}L^4_x}^2.$$
This last inequality is an immediate consequence of \eqref{vintegrada}. The first one can be obtained from the Duhamel formula \eqref{duhamel}. Indeed, by the inhomogeneous Strichartz estimate,
\begin{multline*}\Big\|\int_0^tS(t-s)e^{-s/\mu}v_0u(s)ds\Big\|_{L_t^2L_x^{4}}\lesssim \|e^{-t/\mu}v_0u\|_{L_t^2L^{4/3}_x}\\
\lesssim \|e^{-t/\mu}u\|_{L^2_tL^4_x}\|v_0\|_{L^{\infty}_tL^2_x}\lesssim \|u\|_{L^2_tL^4_x}\|v_0\|_{L^\infty_tL^2_x} \lesssim \|v_0\|_{L^2}\|u\|_{L^2_tL^4_x}$$
\end{multline*}
and
\begin{equation}
\label{N3prim}
\begin{array}{llll}
\displaystyle\Big\|\int_0^tS(t-s)\int_0^se^{-(s-s')/\mu}|u(s')|^2u(s)ds'ds\Big\|_{L_t^2L_x^4}\\
\lesssim\displaystyle\Big\|\int_0^te^{-(t-s')/\mu}|u(s')|^2u(t)ds'\Big\|_{L_t^2L_x^{4/3}}\lesssim \Big\|\int_0^te^{-(t-s')/{\mu}} \|u(s')\|_{L^4}^2\|u(t)\|_{L^4}ds'\Big\|_{L^2_t}\\
\displaystyle\lesssim \|u\|_{L^2_tL^4_x}\int_0^Te^{-(T-s')/\mu} \|u(s')\|_{L^4}^2ds'\leq \|u\|_{L^2_tL^4_x}^3. 
\end{array}
\end{equation}
These estimates imply that
\begin{align*}
\|u\|_{L^2_tL^4_x} &\lesssim \|S(t)u_0\|_{L^2_tL^4_x} +\left\|\int_0^t S(t-s)e^{-s/\mu}v_0u(s)ds \right\|_{L^2_t L^4_x}\\&\quad+\left\|\int_0^tS(t-s)\int_0^se^{-(s-s')/\mu}|u(s')|^2u(s)ds'ds\right\|_{L^2_t L^4_x}\\ &\lesssim \|S(t)u_0\|_{L^2_t L^4_x}+\|u\|_{L^2_t L_x^4}\|v_0\|_{L^2}+\|u\|_{L^2_tL^4_x}^3.
\end{align*}
Now, since $\|v_0\|_{L^2}$ is taken arbitrarely small, \eqref{lwpL2L4} follows. Finally, the estimate \eqref{lwpLinftyL2} can be obtained by analogous computations.
\end{proof}

As a direct consequence of the global well-posedness for small data, we prove Theorem \ref{teo:scat4} in the case $(X,Y)=(L^2(\er^4), L^2(\er^4))$.

\begin{proof}[Proof of Theorem \ref{teo:scat4} for $(X.Y)=(L^2(\er^4), L^2(\er^4))$]
We set $f(t)=S(-t)u(t)$. Since $$\|u\|_{L^2((0,\infty), L^4_x)}<\infty,$$
we have
\begin{align*}
\|f(t)-f(t')\|_{L^2} &= \|S(t)(f(t)-f(t'))\|_{L^2} \\&\lesssim \|u\|_{L^2((t',t), L^4_x)}\|v_0\|_{L^2} + \|u\|_{L^2((t',t), L^4_x)}^3 \to 0,\quad t,t'\to \infty
\end{align*}
Hence there exists $u_+:=\lim_{t\to\infty} S(-t)u(t)\in L^2(\er^4)$. 

For any given $\delta>0$, pick $t'$ so that $\|u\|_{L^2((t',\infty), L^4_x)} < \delta$ and then choose $t$ large enough satisfying $e^{-(t-t')}\|v(t')\|_{L^2} <\delta$. Then
\begin{align*}
\|v(t)\|_{L^2} &\le e^{-(t-t')}\|v(t')\|_{L^2} + \int_{t'}^t e^{-(t-s)}\|u(s)\|_{L^4}^2 ds \\&\le e^{-(t-t')}\|v(t')\|_{L^2} + \|u\|_{L^2((t',t), L^4_x)}^2 <2\delta.
\end{align*}
This implies that $v(t)\to 0$ as $t\to\infty$, which concludes the proof.
\end{proof}

\subsection{$H^1(\er^4)\times H^1(\er^4)$ theory}
%
We start with a simple remark regarding Theorem \ref{bi}: from its proof (see \cite{corchosilva}), it is standard to show the following blow-up alternative: putting 
$$T_*=\sup\{T>0\,:\,(u,v)\in C([0,T];H^1(\er^4)\times H^1(\er^4))\},$$
then
\begin{equation}\label{bua2}
T_*<+\infty \Rightarrow \lim_{T\to T_*}\Big(\|\nabla u\|_{L^{\infty}_TL^2_x}+\|u\|_{L^2_TW^{1,4}_x}+\| v\|_{L^\infty_TH^1_x}\Big)=+\infty.
\end{equation}
Notice that $\|u\|_{L^2}$ is conserved by the flow of the Schr\"odinger-Debye system. By \eqref{vintegrada},
$$\|v\|_{L^{\infty}_{T}L^2_x}\leq \|v_0\|_{L^2}+\int_0^T e^{-(T-s)}\|u(s)^2\|_{L^2}ds\leq \|v_0\|_{L^2}+\|u\|_{L^2_{T}L^4_x}^2,\quad T<T_*,$$
and
\begin{displaymath}
\begin{array}{lllll}
\displaystyle\|\nabla v\|_{L^{\infty}_{T}L^2_x}&\leq&\displaystyle \|\nabla v_0\|_{L^2}+\int_0^T e^{-(T-s)}\|u(s)\nabla u(s)\|_{L^2}ds\\
\\
&\displaystyle\leq&\displaystyle\|\nabla v_0\|_{L^2}+\int_0^T\|u(s)\|_{L^4}\|\nabla u(s)\|_{L^4}ds\\
\\
&\leq&\|\nabla v_0\|_{L^2}+\|\nabla u\|_{L^2_TL^4_x}\|u\|_{L^2_TL^4_x}.
\end{array}
\end{displaymath}
Therefore, we may express the blow-up alternative exclusively in terms of $u$:
\begin{equation}\label{bua}
T_*<+\infty \Rightarrow \lim_{T\to T_*}\Big(\|\nabla u\|_{L^{\infty}_TL^2_x}+\|u\|_{L^2_TW^{1,4}_x}\Big)=+\infty.
\end{equation}

\noindent Using this fact, we will show the following result:

\begin{Prop}[Global existence for small initial data in dimension $d=4$]\label{ge}
There exists $\epsilon_0>0$ such that, if $\|u_0\|_{H^1}+\|v_0\|_{H^1}<\epsilon<\epsilon_0$, the solution $(u(t),v(t))$ given by Theorem \ref{lwp4} is global in time and
$$\|u\|_{L^{\infty}((0,\infty), H^1_x)}+\|u\|_{L^2((0,\infty),W^{1,4}_x)}\leq 2\epsilon.$$
\end{Prop}

\noindent In order to show Proposition $\ref{ge}$, we begin by proving some \textit{a priori} estimates:
\begin{Lemma}\label{ap}Let $u_0,v_0\in H^1(\er^4)$ and $(u(t),v(t))$ the corresponding solution given by Theorem \ref{lwp4}. Then, for any $T<T_*$,
\begin{multline}
\label{ap2}
\|\nabla u\|_{L_T^2L_x^{4}} + \|\nabla u\|_{L^\infty_T L^2_x}\lesssim \|\nabla u_0\|_{L^2}+\|v_0\|_{L^4_x}\|u\|_{L^{\infty}_TH^1_x}+\|v_0\|_{H^1}\|u_0\|_{L^2}
+\|\nabla u\|_{L^2_TL_x^4}\|u\|_{L^2_TL_x^4}^2.
\end{multline}
\end{Lemma}
\begin{proof}
Differentiating \eqref{duhamel},
\begin{multline}
\label{duhamelderivado}
\nabla u(t)=S(t)\nabla u_0+i\int_0^t S(t-s)e^{-s}(u(s)\nabla v_0+v_0 \nabla u(s))ds\\
+2i\lambda \int_0^tS(t-s)\int_0^se^{-(s-s')}Re(\nabla u(s')\overline{u(s')})u(s)ds'\\
+i\lambda \int_0^tS(t-s)\int_0^s e^{-(s-s')}|u(s')|^2\nabla u(s)ds'ds.
\end{multline}
For $(q,r)=(2,4), (\infty,2)$, we have $\|S(t)\nabla u_0\|_{L^q_TL^r_x}\lesssim \|\nabla u_0\|_{L^2}$ and
\begin{multline*}
\Big\|\int_0^t S(t-s)e^{-s}(u(s)\nabla v_0+v_0 \nabla u(s))ds\Big\|_{L^q_TL^r_x}\\
\lesssim \|e^{-t}(\nabla v_0u+v_0\nabla u)\|_{L^2_TL^{4/3}_x}\lesssim \|v_0\|_{L^2}\|u\|_{L^{2}_TW^{1,4}_x}+\|u\|_{L^{\infty}_TL^4_x}\|v_0\|_{H^1}.
\end{multline*}
Moreover,
\begin{multline*}
\Big\|\int_0^tS(t-s)\int_0^se^{-(s-s')}|u(s')|^2\nabla u(s)ds'ds\Big\|_{L_T^qL_x^r}\\\lesssim \Big\|\int_0^te^{-(t-s')}|u(s')|^2\nabla u(t)ds'\Big\|_{L_T^2L_x^{4/3}}\\
\lesssim \Big\|\int_0^te^{-(t-s')} \|u(s')\|_{L^4}^2\|\nabla u(t)\|_{L^4}ds'\Big\|_{L^2_T}\lesssim \|u\|_{L^{2}_TL^4_x}^2\|\nabla u\|_{L^2_TL^4_x},
\end{multline*}
and
\begin{equation}
\label{N2}
\begin{array}{llll}
\displaystyle\Big\|\int_0^tS(t-s)\int_0^se^{-(s-s')}\nabla u(s')\overline{u(s')}u(s)ds'ds\Big\|_{L_T^qL_x^r}^2\\
\\
\displaystyle\lesssim \Big\|\int_0^se^{-(s-s')}\nabla u(s')\overline{u(s')}u(s)ds'\Big\|_{L_T^2L_x^{4/3}}^2\\
\\
\displaystyle\lesssim \int_0^T \Big\|u(s) \Big(\int_0^se^{-(s-s')}\nabla u(s')\overline{u(s')}\|ds'\Big)\Big\|_{L^{4/3}}^2ds\\
\\
\displaystyle\lesssim \int_0^T \|u(s)\|_{L^4}^2\Big(\int_0^se^{-(s-s')}\|\nabla u(s')\|_{L^4}\|u(s')\|_{L^4}ds'\Big)^2ds\\
\\
\displaystyle\lesssim \Big(\int_0^T\|\nabla u(s')\|_{L^4}\|u(s')\|_{L^4}ds'\Big)^2\int_0^T\|u(s)\|_{L^4}^2ds\\
\\
\displaystyle\lesssim \|\nabla u\|_{L^2_TL_x^4}^2\|u\|_{L^2_TL_x^4}^4,
 \end{array}
 \end{equation}
which concludes the proof.
\end{proof}
\medskip
\begin{proof}[Proof of Proposition \ref{ge}]
Putting $h(t)=\|u\|_{L^{\infty}_TH^1_x}+\|u\|_{L^2_TW^{1,4}_x}$, it follows from Lemma \ref{ap} and \eqref{lwpL2L4} that
$$h(t)\leq C(\|u_0\|_{H^1}+h(t)\|v_0\|_{H^1}+h(t)^3),$$
where $C$ is a positive constant. Therefore, choosing $\epsilon_0$ such that $C\epsilon_0\leq \frac 12$, we obtain that
$$h(t)\leq 2C(\|u_0\|_{H^1}+h(t)^3).$$  
\noindent
Thus, for $\|u_0\|_{H^1}<\epsilon$, a classical obstruction argument shows that $h(t)$ remains bounded by $2\epsilon$. In view of the blow-up alternative \eqref{bua}, this implies that the solution is globally defined, which completes the proof of Proposition \ref{ge}.
\end{proof}

As in the $L^2$ case, the global well-posedness result implies scattering in $H^1(\er^4)\times H^1(\er^4)$, thus finishing the proof of Theorem \ref{teo:scat4} for $(X,Y)=(H^1(\er^4), H^1(\er^4))$.

\subsection{$\Sigma(\er^4) \times H^1(\er^4)$ theory}

Let us introduce the so-called vector field operator defined in $\Sigma(\er^4)=H^1(\er^4)\cap L^2(|x|dx)$ by 
$$\Gamma=x+it\nabla.$$
We recall that $[\Gamma, i\partial_t + \frac{1}{2}\Delta]=0$. We begin by showing 
\begin{Lemma}
	\label{Gamma}
	In the conditions of Theorem \ref{ge}, let $u_0\in \Sigma(\er^4)$. If  $\|u_0\|_{\Sigma} + \|v_0\|_{H^1}<\epsilon$ small enough,
	$$\|\Gamma(u)\|_{L^2((0,\infty),L_x^4)}+\|\Gamma(u)\|_{L^\infty((0,\infty),L_x^2)}\lesssim \|u_0\|_{\Sigma} + \|v_0\|_{H^1}.$$
\end{Lemma}
\begin{proof}
From \eqref{SD}, putting $w=\Gamma(u)$,
$$iw_t+\Delta w=\Gamma(uv)=xuv+it(u\nabla v+v\nabla u)=vw+itu\nabla v.$$
By \eqref{vintegrada},
\begin{displaymath}
	\begin{array}{llllll} 
iw_t+\Delta w&=&\displaystyle vw+ite^{-t}u\nabla v_0+\quad itu\lambda \int_0^te^{-(t-s)}\nabla|u(s)|^2ds\\
&=&\displaystyle vw+ite^{-t}u\nabla v_0+iu\lambda \int_0^t(t-s)e^{-(t-s)}\nabla|u(s)|^2ds\\
&&\displaystyle +iu\lambda \int_0^tse^{-(t-s)}\nabla|u(s)|^2ds\\
&=&\displaystyle ite^{-t}u\nabla v_0+vw+iu\lambda \int_0^t(t-s)e^{-(t-s)}\nabla|u(s)|^2ds\\
&&\displaystyle +u\lambda \int_0^te^{-(t-s)}\Big(\Gamma(|u(s)|^2)-x|u(s)|^2\Big)ds\\
&=&\displaystyle iE(-t)u\nabla v_0+vw+iu\lambda \int_0^tE(-(t-s))\nabla|u(s)|^2ds\\
&&\displaystyle +u\lambda \int_0^te^{-(t-s)}\Big(\overline{u}w-u\overline{w}\Big)ds\\
\\
&:=& N_1(x,t)+N_2(x,t)+N_3(x,t)+N_4(x,t),
\end{array}
\end{displaymath}
where we have put $E(t)=-te^{t}$ and used the identity
$$is\nabla(u\overline{u})=is\nabla u\overline{u}+isu\nabla\overline{u}=\overline{u}(\Gamma(u)-xu)+u(-\overline{\Gamma(u)}+x\overline{u})=\overline{u}w-u\overline{w}.$$
\noindent
In integral form, we obtain
\begin{equation}
\label{wiform}
w(x,t)=S(t)w_0(x)+\sum_{i=1}^4\int_0^t S(t-s)N_i(x,s)ds.
\end{equation}
For $(q,r)=(2,4), (\infty,2)$, the homogeneous Strichartz estimate yields
$$\|S(t)w_0\|_{L^q_TL^r_x}\lesssim \|w_0\|_{L^2}\lesssim \|xu_0\|_{L^2}+\|u_0\|_{H^1}.$$

Note that $E(-t)$ is bounded in $[0,+\infty[$,  $\displaystyle \lim_{t\to \infty} E(-t)=0$ and $\displaystyle \int_0^t E(-(t-s))ds=1-(t+1)e^{-t}$, which remains bounded as $t\to\infty$. Hence
$$
\left\|\int_0^t S(t-s)N_1(x,s)ds\right\|_{L^q_TL^r_x} \lesssim \|u\nabla v_0\|_{L^2_T L^{4/3}_x}\lesssim \|u\|_{L^2_TL^4_x}\|v\|_{L^\infty_T L^2_x}.
$$
Also,
$$\Big\|\int_0^t S(t-s)N_2(x,s)ds\Big\|_{L^q_TL^r_x}\lesssim\|vw\|_{L^2_TL^{4/3}_x}\lesssim \Big(\int_0^T\|w(s)\|_{L^4}^2\|v(s)\|_{L^2}^2ds\Big)^{\frac 12}$$
$$\lesssim \|v\|_{L^{\infty}_TL^2_x}|w\|_{L^2_TL^4_x}.$$
Following the steps of estimate \eqref{N2},
$$\Big\|\int_0^t S(t-s)N_3(x,s)ds\Big\|_{L^q_TL^r_x}\lesssim  \|\nabla u\|_{L^2_TL_x^4}\|u\|_{L^2_TL_x^4}^2.$$
Moreover, in view of \eqref{N3prim},
$$\Big\|\int_0^t S(t-s)N_4(x,s)ds\Big\|_{L^q_TL^r_x}\lesssim \|u\|_{L^2_TL_x^4}^2\|w\|_{L^2_TL_x^4}.$$
We finally arrive at
\begin{align*}
\|\Gamma(u)\|_{L^2_TL_x^4}+\|\Gamma(u)\|_{L^\infty_TL_x^2}\lesssim&\|xu_0\|_{L^2}+\|u_0\|_{H^1}+\|u\|_{L^2_TL^4_x}\|v\|_{L^\infty_T L^2_x}+\|v\|_{L^{\infty}_TL^2_x}\|\Gamma(u)\|_{L^2_TL^4_x}\\
+&\|\nabla u\|_{L^2_TL_x^4}\|u\|_{L^2_TL_x^4}^2+\|u\|_{L^2_TL_x^4}^2\|\Gamma(u)\|_{L^2_TL_x^4}.
\end{align*}
In view of Proposition \ref{ge}, choosing $(u_0,v_0)$ small enough such that $$\|u\|_{L^2((0,\infty), L_x^4)}^2+\|v\|_{L^{\infty}((0,\infty),L^2_x)}<\frac 12,$$ we obtain, as stated in Lemma \ref{Gamma},
$$\|\Gamma(u)\|_{L^2_TL_x^4} + \|\Gamma(u)\|_{L^\infty_TL_x^2}\lesssim\|u_0\|_{\Sigma}+\|v_0\|_{H^1},\quad T>0.$$
\end{proof}

\begin{proof}[Proof of Theorem \ref{teo:scat4} in the case $(X,Y)=(\Sigma(\er^4), H^1(\er^4))$]
The convergence \eqref{eq:scattering} follows easily from the previous lemma. The decay \eqref{co2} is obtained by Gagliardo-Nirenberg's inequality: indeed, writing $z=e^{\frac{-i|x|^2}{2t}} u$, we have $\Gamma(u)=it e^{\frac{i|x|^2}{2t}}\nabla z$. Then, by the previous lemma,
$$
\|\nabla z(t)\|_{L^2} = \frac{1}{t}\|\Gamma(u)(t)\|_{L^2}\lesssim \frac{\|u_0\|_{\Sigma}+\|v_0\|_{H^1}}{t}.
$$
Hence
$$
\|u(t)\|_{L^p}=\|z(t)\|_{L^p}\lesssim \|z(t)\|_{L^2}^{1-d\left(\frac{1}{2}-\frac{1}{p}\right)} \|\nabla z(t)\|_{L^2}^{d\left(\frac{1}{2}-\frac{1}{p}\right)} \lesssim  \frac{\|u_0\|_{\Sigma}+\|v_0\|_{H^1}}{t^{d\left(\frac{1}{2}-\frac{1}{p}\right)}}.
$$
\end{proof}
\section{Scattering theory in dimensions $d=2,3$}
\noindent
In this section, we prove Theorem \ref{scat23}. We start with the three-dimensional case.
\subsection{Proof of Theorem \ref{scat23} in dimension $d=3$}
Contrarely to the fourth dimensional case studied in the previous section, in order to obtain a scattering result in dimension $3$, we must first prove some time decay for the solutions of \eqref{SD}:
\begin{Prop}
	\label{decay}
Let $(u_0,v_0)\in \Sigma(\er^3)\times H^1(\er^3)$ and $(u(t),v(t))$ the corresponding solution to (\ref{SD}) given by Theorem \ref{bi}. Then, there exists $\delta>0$ such that, if $\|u_0\|_{\Sigma}+\|v_0\|_{L^2}<\delta$,
$$\|u\|_{L^{\infty}((0,\infty),H^1_x)}+\|u\|_{L^{\frac 83}((0,\infty),W^{1,4}_x)}<2\delta$$
and, for all $t>0$,
\begin{equation}
\label{decayR3}
\|u(t)\|_{L^4}\lesssim \frac 1{t^{\frac 34}}.
\end{equation}
\end{Prop}
\noindent
Before proving this result, we will need the following two preliminary lemmas:
\begin{Lemma}
\label{l1}	
	 Let  $(u_0,v_0)\in H^1(\er^3)\times H^1(\er^3)$ and $(u(t),v(t))$ the corresponding solution to (\ref{SD}) given by Theorem \ref{bi}. Then, there exists $\epsilon>0$ such that, whenever $\|u_0\|_{H^1}+\|v_0\|_{H^1}<\delta$, $\delta$ small, and, for some $T>1$,
$$\vertiii{u}_T:=\sup_{0<t<T}\{t^{\alpha}\|u(t)\|_{L^4_x}\}\leq \epsilon,\quad \alpha>\frac 58,$$
then
$$\|u\|_{L^{\infty}_TH^1_x}+\|u\|_{L^{\frac 83}_TW^{1,4}_x}<2\delta.$$

\end{Lemma}
\begin{proof}
Using the Duhamel formula \eqref{duhamel}, we begin by estimating $u$  in the Strichartz norms $L^{q}_TL^r_x$, $(q,r)=(\frac 83,4),(\infty,2)$:

$$
\Big\|\int_0^tS(t-s)e^{-s}v_0u(s)ds\Big\|_{L_T^{ q}L_x^{r}}\lesssim \|e^{-t}v_0u(t)\|_{L_T^{\frac 85}L^{4/3}_x}\lesssim \| \|e^{-t}u\|_{L^4}\|v_0\|_{L^2} \|_{L_T^{\frac 85}}
$$
and
\begin{equation}
\label{Stru1}
\Big\|\int_0^tS(t-s)e^{-s}v_0u(s)ds\Big\|_{L_T^{q}L_x^{r}}\lesssim \|v_0\|_{L^2}\|u\|_{L_T^{\frac 83}L_x^4},
\end{equation}
where we have used the H\"older inequality. Also,
\begin{equation}
\label{N3primbis}
\begin{array}{llll}
\displaystyle\Big\|\int_0^tS(t-s)\int_0^se^{-(s-s')}|u(s')|^2u(s)ds'ds\Big\|_{L_T^{q}L_x^r}\\
\\
\lesssim\displaystyle\Big\|\int_0^te^{-(t-s')}\|u(s')\|_{L^4}^2\|u(t)\|_{L^4}ds'\Big\|_{L_T^{\frac 85}L_x^4}\\
\\
\displaystyle \lesssim \|u\|_{L_T^{\frac 83}L_x^4}\Big\|\int_0^te^{-(t-s')}\|u(s')\|_{L^4}^2ds' \Big\|_{L^4_T}
\end{array}
\end{equation}
We now split the time interval $[0;T]$:
$$\Big\|\int_0^te^{-(t-s')}\|u(s')\|_{L^4}^2ds' \Big\|_{L^4(0,1)}\lesssim \|u\|_{L_T^{\infty}L_x^4}^2\lesssim \|u\|_{L_T^{\infty}H^1_x}^2$$
and
$$\Big\|\int_0^te^{-(t-s')}\|u(s')\|_{L^4}^2ds' \Big\|_{L^4(1,T)}\leq \Big\|\int_0^{t/2}e^{-(t-s')}\|u(s')\|_{L^4}^2ds' \Big\|_{L^4(1,T)}$$
$$+\Big\|\int_{t/2}^te^{-(t-s')}\|u(s')\|_{L^4}^2ds' \Big\|_{L^4(1,T)}$$
$$\leq \|u\|_{L_T^{\infty}H_x^1}+\vertiii{u}_T^2\Big\|\int_{t/2}^t\frac 1{{s'}^{2\alpha}}ds' \Big\|_{L^4(1,T)}$$
$$\leq  \|u\|_{L_T^{\infty}H^1_x}+\vertiii{u}_T^2\Big(\int_1^T\frac 1{{t}^{4(2\alpha-1)}}dt \Big)^4 $$
and this last integral is bounded for $\alpha>\frac 58$.
Hence,
\begin{equation}
\Big\|\int_0^tS(t-s)\int_0^se^{-(s-s')}|u(s')|^2u(s)ds'ds\Big\|_{L_T^{q}L_x^r}\lesssim \|u\|_{L_T^{\frac 83}L_x^4}(\|u\|_{L_T^{\infty}H^1_x}^2+\vertiii{u}_T^2).
\end{equation}
Combining this last inequality with \eqref{Stru1}, we obtain
\begin{equation}
\label{parte1}
\|u\|_{L^{q}_TL^r_x}\lesssim \|u_0\|_{L^2}+\|u\|_{L_T^{\frac 83}L_x^4}(\|v_0\|_{L^2}+\|u\|_{L_T^{\infty}H^1_x}^2+\vertiii{u}_T^2)
\end{equation}
Next, in view of \eqref{duhamelderivado}, we estimate $\|\nabla u\|_{L^{q}_TL^r_x}$. As in \eqref{Stru1},
\begin{equation}\Big\|\int_0^t S(t-s)e^{-s}(u(s)\nabla v_0+v_0 \nabla u(s))ds\Big\|_{L^{q}_TL^r_x}\lesssim \|\nabla v_0\|_{L^2}\|u\|_{L_T^{\frac 83}L_x^4}+\|v_0\|_{L^2}\|\nabla u\|_{L_T^{\frac 83}L_x^4}.
\end{equation}
Also, as in \eqref{N3primbis}, 
 \begin{equation}
\Big\|\int_0^tS(t-s)\int_0^se^{-(s-s')}|u(s')|^2\nabla u(s)ds'ds\Big\|_{L_T^{q}L_r^4}\lesssim \|\nabla u\|_{L_T^{\frac 83}L_x^4}(\|u\|_{L_T^{\infty}H^1_x}^2+\vertiii{u}_T^2).
\end{equation}
Furthermore, by H\"older,

$$\Big\|\int_0^tS(t-s)\int_0^se^{-(s-s')}Re(\nabla u(s')\overline{u(s')})u(s)ds'\Big\|_{L_T^{q}L_x^r}$$
$$\lesssim \Big\|\|u(t)\|_{L^4}\int_0^te^{-(t-s')}\|\nabla u(s')\|_{L^4}\|u(s')\|_{L^4}ds'\Big\|_{L_T^{\frac 85}}$$
$$\lesssim \|\nabla u\|_{L_T^{\frac 83}L_x^4} \Big\|\|u(t)\|_{L^4}\Big(\int_0^te^{-\frac 85(t-s')}\|u(s')\|_{L^4}^{\frac 85}ds'\Big)^{\frac 58}\Big\|_{L_T^{\frac 85}}.$$
Now,
$$\Big\|\|u(t)\|_{L^4}\Big(\int_0^te^{-\frac 85(t-s')}\|u(s')\|_{L^4}^{\frac 85}ds'\Big)^{\frac 58}\Big\|_{L^{\frac 85}(0,1)}$$
$$\lesssim \|u\|_{L_T^{\frac 83}L_x^4} \Big\|\Big(\int_0^te^{-\frac 85(t-s')}\|u(s')\|_{L^4}^{\frac 85}ds'\Big)^{\frac 58}\Big\|_{L^{4}(0,1)}\lesssim \|u\|_{L_T^{\frac 83}L_x^4}\|u\|_{L_T^{\infty}H_x^1}$$
and, similarly
$$\Big\|\|u(t)\|_{L^4}\Big(\int_0^{\frac t2}e^{-\frac 85(t-s')}\|u(s')\|_{L^4}^{\frac 85}ds'\Big)^{\frac 58}\Big\|_{L^{\frac 85}(0,1)}\lesssim \|u\|_{L_T^{\frac 83}L_x^4}\|u\|_{L_T^{\infty}H_x^1}.$$
Finally, using the time decay, 
$$\Big\|\|u(t)\|_{L^4}\Big(\int_{\frac t2}^te^{-\frac 85(t-s')}\|u(s')\|_{L^4}^{\frac 85}ds'\Big)^{\frac 58}\Big\|_{L^{\frac 85}(1,T)}
\lesssim \vertiii{u}_T^2\int_1^T\frac 1{s^{\frac {8\alpha}5}}\frac 1{s^{\frac{8\alpha}5-1}}ds\lesssim \vertiii{u}_T^2$$
for $\alpha>\frac 58$.
Hence 
\begin{multline}
\label{parte2}
\|\nabla u\|_{L_T^{q}L_x^r}\lesssim \|\nabla u_0\|_{L^2}+ \|\nabla v_0\|_{L^2}\|u\|_{L_T^{\frac 83}L_x^4}+\|v_0\|_{L^2}\|\nabla u\|_{L_T^{\frac 83}L_x^4}+\\
 \|\nabla u\|_{L_T^{\frac 83}L_x^4}(\|u\|_{L_T^{\infty}H^1}^2+\vertiii{u}_T^2)+\|\nabla u\|_{L_T^{\frac 83}L_x^4}(\|u\|_{L^{\infty}_xH_x^1}^2+\vertiii{u}_T^2).
\end{multline}
Combining this inequality with \eqref{parte1} and putting $h(T)=\|u\|_{L^{\infty}((0,T),H^1_x)}+\|u\|_{L^{\frac 83}((0,T),W^{1,4}_x)}$ we obtain, for small $\delta$ and $\epsilon$,
$$h(t)\lesssim \delta+h(t)^3,$$
and the Lemma is proved by using once again an obstruction argument.
\end{proof}
\bigskip
\begin{Lemma}
\label{l2}
In the conditions of Lemma \ref{l1}, if additionally $xu_0\in L^2(\er^3)$ and $\|u_0\|_{\Sigma}+\|v_0\|_{L^2}<\delta$, then, for all $t\leq T$,
$$\|u(t)\|_{L^4}\lesssim \frac{\delta}{t^{\frac 34}}.$$
\end{Lemma}
\begin{proof}
	Arguing as in the proof of Theorem \ref{teo:scat4} in the case $(X,Y)=(\Sigma(\er^4), H^1(\er^4))$, we only need to show that $\|\Gamma(u)\|_{L^{\infty}_TL^2_x}<\delta$, which can be obtained from an \emph {a priori} estimate as in Lemma \ref{Gamma}. This can be easily obtained by using similar computations as in the proof of Lemma \ref{l1} to estimate the $L_T^{\infty}L^2_x$ norm of the right-hand-side of \eqref{wiform}.
\end{proof}

\bigskip
\noindent We are now able to show Theorem \ref{scat23} in the case $d=3$:
\begin{proof}[Proof of Theorem \ref{decay}]
	Let $\delta>0$ and $(u_0,v_0)$ such that $\|u_0\|_{\Sigma}+\|v_0\|_{H^1}<\delta$. Define $T^*=\sup\{T\,:\,\vertiii{u}_T<\epsilon\}$, for a fixed $\epsilon>0$ small enough as in Lemma \ref{l1}. Our goal is to  establish that $T^*=+\infty$ provided that $\delta>0$ is small enough. By contradiction, we  assume that $T^*<+\infty$.

\smallskip
\noindent
From the local well-posedness result, we have, for $\delta$ small enough,
$$\vertiii{u}_{T=1}\lesssim \|u\|_{L^{\infty}((0,1),L_x^4)} \lesssim \|u\|_{L^{\infty}((0,1),H_x^1)}\lesssim \|u_0\|_{H^1}+ \|v_0\|_{H^1}\lesssim \delta$$
which implies that $\vertiii{u}_{T=1}<\epsilon$, hence $T^*>1$.

	\medskip
	\noindent
	On the other hand, from Lemma \ref{l2}, ${T^*}^{\frac 34}\|u(T^{*})\|_{L^4}\lesssim\delta.$ Also, by definition of $T^*$ and by continuity of $t\to u(\cdot,t)\in L^4$, 
	${T^*}^{\alpha}\|u(T^*)\|_{L^4}=\epsilon.$
		Hence, $\delta\gtrsim{T^*}^{\frac 34-\alpha}\gtrsim 1$ for $\frac 58<\alpha<\frac 34$, which is absurd for small $\delta$, and the proof of Theorem \ref{scat23} for $d=3$ follows.
\end{proof}
\subsection{Proof of Theorem \ref{scat23} in dimension $d=2$}
The proof of the scattering result in dimension $d=2$ follows the lines of the three dimensional case. However, the optimal linear asymptotic decay, for $u_0\in\Sigma(\er^2)$, reads
$$\|S(t)u_0\|_{L^4}\sim \frac 1 {t^{\frac 12}},$$
hence we must prove a version of Lemma \ref{l1} valid for $\alpha=\frac 12$:
\begin{Lemma}
	\label{l1bis}	
	Let  $(u_0,v_0)\in H^1(\er^2)\times H^1(\er^2)$ and $(u(t),v(t))$ the corresponding solution to (\ref{SD}) given by Theorem \ref{bi}.  Then, there exists $\epsilon>0$ such that, whenever $\|u_0\|_{H^1}+\|v_0\|_{H^1}<\delta$, $\delta$ small, and, for some $T>1$,
	$$\vertiii{u}_T:=\sup_{0<t<T}\{t^{\alpha}\|u(t)\|_{L^4_x}\}\leq \epsilon,\quad \alpha>\frac 38,$$
	then
	$$\|u\|_{L^{\infty}_TH^1_x}+\|u\|_{L^{4}_TW^{1,4}_x}<2\delta.$$
\end{Lemma}

\noindent
Unfortunately, the splitting of the time interval $[0,t]$ into $[0,t/2]$ and $[t/2;t]$ will no longer provide, in this dimension, adequate estimates for the $L_T^qL_x^r$ Strichartz norms $(q,r)=(4,4),(\infty,2)$ of the cubic integrals
$$\nabla^{(i)}\int_0^tS(t-s)\int_0^se^{-(s-s')}|u(s')|^2u(s)ds'ds,\quad i=0,1.$$
Hence, using a different technique, we will show the following Lemma, from which Lemma \ref{l1bis} and Theorem \ref{scat23} follow:   
\begin{Lemma} Let $\alpha>\frac 38$ and $(q,r)=(4,4),(\infty,2)$. Then
$$\Big\|\int_0^tS(t-s)\int_0^se^{-(s-s')}|u(s')|^2u(s)ds'ds  \Big\|_{L_T^qL_x^r}\lesssim \|u\|_{L_T^4L_x^4}\vertiii{u}^2_T$$
and 
$$\Big\|\nabla\int_0^tS(t-s)\int_0^se^{-(s-s')}|u(s')|^2u(s)ds'ds  \Big\|_{L_T^qL_x^r}$$
$$\lesssim \|u\|_{L_T^4W^{1,4}_x}\|u\|_{L_T^{\infty}H^1_x}^2+\|\nabla u\|_{L_T^4L_x^4}\vertiii{u}_T^2.$$
\end{Lemma}
\begin{proof}
	By the inhomogeneous Strichartz estimate,
	\begin{multline}
	\Big\|\nabla^{(i)}\int_0^tS(t-s)\int_0^se^{-(s-s')}|u(s')|^2u(s)ds'ds  \Big\|_{L_T^qL_x^r}\lesssim 
	\Big\|\nabla^{(i)}\int_0^se^{-(s-s')}|u(s')|^2u(s)ds'\Big\|_{L_T^{\frac 43}L_x^{\frac 43}}.
	\end{multline}
	For $i=0$,
$$\Big\|\int_0^se^{-(s-s')}|u(s')|^2u(s)ds'\Big\|_{L_T^{\frac 43}L_x^{\frac 43}}\leq \Big\|\int_0^se^{-(s-s')}\|u(s')\|_{L^4}^2\|u(s)\|_{L^4}ds'\Big\|_{L_T^{\frac 43}}$$
$$\leq \|u\|_{L_T^4L_x^4}\Big\|\int_0^se^{-(s-s')}\|u(s')\|_{L^4}^2ds'\Big\|_{L_T^{2}}.$$
We now apply the Paley-Littlewood-Sobolev inequality:
\begin{equation}
\label{pls}
\Big\|\int_0^se^{-(s-s')}\|u(s')\|_{L^4}^2ds'\Big \|_{L_T^{2}}\leq \Big\|e^{-|s|}\ast\frac 1{|s|^{2\alpha}}\Big\|_{L^2}\vertiii{u}_T^2\lesssim \vertiii{u}_T^2
\end{equation}
for $\displaystyle 1+\frac 12=\frac 1p+2\alpha,\ p>1$, provided $e^{-|s|}\in L^p$. Since it is the case for all $p$, the estimate \eqref{pls} holds for all $\alpha>\frac 14$.\\
Now, for $i=1$,
$$\Big\|\nabla\int_0^se^{-(s-s')}|u(s')|^2u(s)ds'\Big\|_{L_T^{\frac 43}L_x^{\frac 43}}\leq \Big\|\int_0^se^{-(s-s')}\|u(s')\|_{L^4}^2\|\nabla u(s)\|_{L^4}ds'\Big\|_{L_T^{\frac 43}}$$
$$+\Big\|\int_0^se^{-(s-s')}\|u(s')\|_{L^4}\|\nabla u(s')\|_{L^4}\|\nabla u(s)\|_{L^4}ds'\Big\|_{L_T^{\frac 43}}.$$
The first norm can be estimated exacly as in the previous situation.\\
We split the second norm in two parts:
$$\Big\|\int_0^se^{-(s-s')}\|u(s')\|_{L^4}\|\nabla u(s')\|_{L^4}\|\nabla u(s)\|_{L^4}ds'\Big\|_{L_T^{\frac 43}}$$
$$\leq \Big\|\int_0^se^{-(s-s')}\|u(s')\|_{L^4}\|\nabla u(s')\|_{L^4}\|\nabla u(s)\|_{L^4}ds'\Big\|_{L^{\frac 43}(0,1)}$$
$$+\Big\|\int_0^se^{-(s-s')}\|u(s')\|_{L^4}\|\nabla u(s')\|_{L^4}\|\nabla u(s)\|_{L^4}ds'\Big\|_{L_T^{\frac 43}(1,T)}.$$
On one hand,
$$\Big\|\int_0^se^{-(s-s')}\|u(s')\|_{L^4}\|\nabla u(s')\|_{L^4}\|\nabla u(s)\|_{L^4}ds'\Big\|_{L^{\frac 43}(0,1)}$$
$$\leq \|\nabla u\|_{L_T^4L_x^4}\Big\|\Big(\int_0^se^{-\frac 43(s-s')}\|u(s')\|_{L^4}^{\frac 43} \Big)^{\frac 34}\|u(s)\|_{L^4}\Big\|_{L^{\frac 43}(0,1)}$$
$$\lesssim \|\nabla u\|_{L_T^4L_x^4}\| u\|_{L_T^{\infty}L_x^4}^2\lesssim \|\nabla u\|_{L_T^4L_x^4}\| u\|_{L_T^{\infty}H_x^1}^2.$$
\smallskip
\noindent
On the other hand,
$$\Big\|\int_1^se^{-(s-s')}\|u(s')\|_{L^4}\|\nabla u(s')\|_{L^4}\| u(s)\|_{L^4}ds'\Big\|_{L^{\frac 43}(1,T)}$$
$$\lesssim \|\nabla u\|_{L_t^4L_x^4}\Big\|\| u(s)\|_{L^4}\Big(\int _1^se^{-\frac 43(s-s')}\|u(s')\|_{L^4}^{\frac 43}  \Big)^{\frac 34}\Big\|_{L^{\frac 43}(1,T)}$$
$$\lesssim \|\nabla u\|_{L_t^4L_x^4}\Big\|\Big(e^{-\frac 43|s'|}\ast \frac 1{{|s'|}^{\frac 43\alpha}}\Big)^{\frac 34}(s)\frac 1{s^{\alpha}}\Big\|_{L^\frac 43(1,T)}$$
and, by H\"older, for some $2<\delta<4$ to be chosen later,
$$\Big\|\Big(e^{-\frac 43|s'|}\ast \frac 1{{|s'|}^{\frac 43\alpha}}\Big)^{\frac 34}(s)\frac 1{s^{\alpha}}\Big\|_{L^\frac 43(1,T)}\leq \Big\| e^{-\frac 43|s'|}\ast \frac 1{{|s'|}^{\frac 43\alpha}}\Big\|_{L^{\frac {3\delta}{3\delta-4}}}\Big\|\frac 1{s^{\alpha}} \Big\|_{L^{\delta}(1,T)} $$
which is bounded for $\alpha>\max\{\frac 1\delta,\frac 34-\frac 1\delta\}$ and the lemma is proved by taking $\delta=\frac 83$.
\end{proof}
\section{(Modified) scattering in dimension $d=1$}
\noindent
In this section, we show Theorem \ref{modified}, which states a modified scattering result in dimension $d=1$. As stated in the introduction, we deeply rely on the ideas in \cite{Pusateri} which, in turn, derive from the works of Germain, Masmoudi and Shatah (\cite{germain}) concerning the space-time resonances method.

For $(u_0,v_0)\in \Sigma(\er)\times H^1(\er)$ we consider the solution $(u,v)$ to the Cauchy problem (\ref{SD}).
We begin by rewriting the integral version of \eqref{SD} for initial data taken at $t=1$:
\begin{multline*}
u(x,t)=S(t-1)u(x,1)+i\int_1^tS(t-s)u(s)v(s)ds\\
=S(t-1)u(x,1)+i\int_1^tS(t-s)u(s)\Big(e^{-(s-1)}v(x,1)+\int_1^se^{-(s-s')}|u(s')|^2ds'\Big)ds,
\end{multline*}
or, putting $(u_*,v_*):=(e^{-i\frac{\Delta}2}u(1),ev(1))$,
\begin{equation}
\label{duhamel1}
u(x,t)=S(t)u_*(x)+i\int_1^tS(t-s)u(s)\Big(e^{-s}v_*(x)+\int_1^se^{-(s-s')}|u(s')|^2ds'\Big)ds.
\end{equation}
We now define the profile $$f(t):=S(-t)u(t).$$ 
\noindent
The proof of Theorem \ref{modified} relies essentially on the following Proposition (and its proof):
\begin{Propriedade}
	\label{modifiedapriori}
	
	Let $(u_0,v_0)\in \Sigma(\er)\times H^1(\er)$ and $(u,v)$ the solution to  (\ref{SD}) given by Theorem 3.2. Then, putting
	$$\|u\|_{Y_T}:=\|t^{\frac 12} u\|_{L_T^{\infty}L_x^{\infty}}+\|t^{-\alpha} u\|_{L_T^{\infty}H_x^{1}}+\|t^{-\alpha}xf\|_{L_T^{\infty}L_x^{2}}+\|u\|_{L_T^{\infty}L_x^2},$$
	for all $T>0$,
	$$\|u\|_{Y_T}\lesssim \|u_*\|_{\Sigma}+\|v_*\|_{H^1}\|u\|_{Y_T}+\|u\|_{Y_T}^3.$$
\end{Propriedade}
\begin{proof}
By applying $S(-t)$ to \eqref{duhamel1} and taking the Fourier transform,
\begin{equation}
\hat{f}(\xi)=\hat{u_*}+i\int_1^t e^{-i\frac 12\xi^2s} e^{-s}\widehat{u(s)v_*(x)}(\xi)ds+i\int_1^t\int_1^se^{-i\frac 12\xi^2s}e^{-(s-s')}\widehat{|u(s')|^2u(s)}(\xi)ds'ds,
\end{equation}
that is
\begin{equation}
\label{fchapeu3}
\begin{array}{lllll}
\hat{f}(\xi)&=&\displaystyle \hat{u_*}+i\int_1^t e^{-i\frac 12\xi^2s} e^{-s}\int{\hat{u}(\xi-\eta,s)\hat{v_*}(\eta)}d\eta ds\\
&&\displaystyle +i\int_1^t\int_1^s\iint e^{-i\frac 12\xi^2s}e^{-(s-s')}\hat{u}(\sigma,s')\overline{\hat{u}(\sigma-\eta,
	s')} \hat{u}(\xi-\eta,s)d\eta d\sigma ds'ds\\
&=&\displaystyle \hat{u_*}+i\int_1^t\int e^{-i\frac 12(\xi^2-(\xi-\eta)^2)s} e^{-s}{\hat{f}(\xi-\eta,s)\hat{v_*}(\eta,s)}d\eta ds\\
&&\displaystyle +i\int_1^t\int_1^s\iint e^{i\phi}e^{-(s-s')}\hat{f}(\sigma,s')\overline{\hat{f}(\sigma-\eta,
	s')} \hat{f}(\xi-\eta,s)d\eta d\sigma ds'ds,
\end{array}
\end{equation}
where
\begin{displaymath}
\begin{array}{llll}
\phi=\phi(\xi,\eta,\sigma)&=&-\frac 12(\xi^2s+\sigma^2s'-(\sigma-\eta)^2s'+(\xi-\eta)^2s)\\
&=&\eta(\sigma-\xi)s+(\sigma\eta-\frac 12\eta^2)(s-s').
\end{array}
\end{displaymath}
Finally, putting
$$F(s,\xi,\sigma,\eta)=\int_1^se^{i\frac 12(\eta^2-2\eta\sigma)(s-s')}e^{-(s-s')}\hat{f}(\sigma,s')\overline{\hat{f}(\sigma-\eta,
	s')} \hat{f}(\xi-\eta,s)ds',$$
we obtain 
\begin{equation}
\label{fchapeu2}
\begin{array}{lllll}
\hat{f}(\xi)&=&\displaystyle \hat{u_*}+i\int_1^t\int e^{i\frac 12(\eta^2-2\xi\eta)s} e^{-s}{\hat{f}(\xi-\eta,s)\hat{v_*}(\eta,s)}d\eta ds\\
&&\displaystyle +i\int_1^t\int_1^s\iint e^{i\eta(\sigma-\xi)s}F(s,\xi,\sigma,\eta)dsd\eta d\sigma,
\end{array}
\end{equation}
that is, by a change of variables in the last integral,
\begin{equation}
\label{fchapeu}
\begin{array}{lllll}
\hat{f}(\xi)&=&\displaystyle \hat{u_*}+i\int_1^t\int e^{\frac 12i(\eta^2-2\xi\eta)s} e^{-s}{\hat{f}(\xi-\eta,s)\hat{v_*}(\eta,s)}d\eta ds\\
&&\displaystyle +i\int_1^t\int_1^s\iint e^{-i\eta\sigma s}F(s,\xi,\xi-\sigma,\eta)dsd\eta d\sigma.
\end{array}
\end{equation}
We now use these formulae to estimate the various norms appearing in $\|\cdot\|_{Y_T}$.

\medskip
\noindent {\bf Estimate for $\|t^{-\alpha}\partial_xu\|_{L_t^{\infty}L_x^2}$}:
From \eqref{duhamel1} we obtain
$$\|\partial_xu\|_{L^2}\leq \|\partial_x u_*\|_{2}+\int_1^te^{-s}\|\partial_x(uv_*)\|_2ds+\int_1^t\int_1^se^{-(s-s')}\|\partial_x(u^2(s')u(s))\|_{L^2}ds'.$$ 
Also,
$$\int_1^te^{-s}\|\partial_x(uv_*)\|_2ds\leq \int_1^te^{-s}(\|v_*\|_{\infty}\|\partial_xu\|_2+\|\partial_xv_*\|_{2}\|u\|_\infty)ds$$$$\lesssim \|v_*\|_{H^1}\int_1^te^{-s}\|u\|_{H^1}ds\lesssim \|v_*\|_{H^1}\|u\|_{Y_T}\int_1^ts^{\alpha}e^{-s}ds\lesssim \|v_*\|_{H^1}\|u\|_{Y_T}.$$
Now, since
$$\|\partial_x(u^2(s')u(s))\|_{L^2}\lesssim \|u(s')\|_{\infty}^2\|\partial_xu(s)\|_{L^2}+ \|u(s')\|_{\infty}\|u(s)\|_{\infty}\|\partial_xu(s')\|_{L^2},$$
we obtain
\begin{align*}
\int_1^t\int_1^se^{-(s-s')}\|\partial_x(u^2(s')u(s))\|_{L^2}ds'ds \lesssim \int_1^t\|\partial_xu(s)\|_{L^2}\int_1^se^{-(s-s')}\|u(s')\|_{\infty}^2ds'ds\\+\int_1^t\|u(s)\|_{\infty}\int_1^se^{-(s-s')}\|u(s')\|_{\infty}\|\partial_xu(s')\|_{2}ds'ds \\\lesssim  \Big(\int_1^ts^{\alpha}\int_1^se^{-(s-s')}{s'}^{-1}ds'ds+\int_1^ts^{-\frac 12}\int_1^se^{-(s-s')}{s'^{\alpha-\frac 12}}ds'ds\Big)\|u\|_{Y_T}^2
\end{align*}

In order to estimate these integrals, notice that for all $\beta\in\er$ and $\gamma>0$, 
$$\displaystyle\lim_{s\to \infty}\frac{\int_1^s\frac{e^{\gamma s'}}{{s'}^{\beta}}ds'}{\frac{e^{\gamma s}}{s^{\beta}}}=\frac 1{\gamma},$$
hence for all $\beta$ and for all $s\geq 1$,
\begin{equation}
\label{estfant}
\int_1^s\frac{e^{\gamma s'}}{{s'}^{\beta}}ds'\lesssim \frac{ e^{\gamma s}}{s^{\beta}}.
\end{equation}
Hence,
\begin{equation}
\int_1^ts^{\alpha}\int_1^se^{-(s-s')}{s'}^{-1}ds'ds+\int_1^ts^{-\frac 12}\int_1^se^{-(s-s')}{s'^{\alpha-\frac 12}}ds'ds
\lesssim \int_1^t s^{\alpha-1}ds\lesssim t^{\alpha}.
\end{equation}
Finally, 
\begin{equation}
\label{estimativab}
\|t^{-\alpha}\partial_x u\|_{L_t^{\infty}L^2_x}\lesssim \|\partial_xu_*\|_{L^2}+\|v_*\|_{H^1}\|u\|_{Y_T}+\|u\|_{Y_T}^3.
\end{equation}

\bigskip
\noindent 
{\bf Estimate for $\|t^{-\alpha}xf\|_{L_t^{\infty}L_x^2}$}:
From \eqref{fchapeu},
\begin{multline*}
\|xf\|_{L^2}=\|{\partial_{\xi}\hat{f}}\|_{L^2}\lesssim \|\partial_{\xi}\widehat{u_*}\|_{L^2}\\
+\displaystyle \Big\|\int_1^t\int \partial_{\xi}\Big( e^{i\frac 12(\eta^2-2\xi\eta)s} e^{-s}{\hat{f}(\xi-\eta,s)\hat{v_*}(\eta)}\Big)d\eta ds\Big\|_{L^2}
\\
\displaystyle +\Big\|\int_1^t\int_1^s\iint \partial_{\xi}\Big(e^{-i\eta\sigma s}F(s,\xi,\xi-\sigma,\eta)\Big)dsd\eta d\sigma\Big\|_{L^2}
\end{multline*}
\begin{multline}
\label{inter}
\lesssim \|xu_*\|_{L^2}+\displaystyle\displaystyle \Big\|\int_1^t\int \Big( e^{i\frac 12(|\xi-\eta|^2-\xi^2)s} e^{-s}{\hat{f}(\xi-\eta,s)\eta\hat{v_*}(\eta)}\Big)d\eta ds\Big\|_{L^2}\\
+\displaystyle\displaystyle \Big\|\int_1^t\int \Big( e^{i\frac 12(|\xi-\eta|^2-\xi^2)s} e^{-s}{\widehat{[x f]}(\xi-\eta,s)\hat{v_*}(\eta)}\Big)d\eta ds\Big\|_{L^2}\\
+\displaystyle \Big\|\int_1^t\int_1^s\iint \Big(e^{-i\eta\sigma s}\partial_{\xi}F(s,\xi,\xi-\sigma,\eta)\Big)dsd\eta d\sigma\Big\|_{L^2}\\
\lesssim \|xu_*\|_{L^2}+I+II+III
\end{multline}
We proceed with the estimation of the three integrals in the right-hand-side of \eqref{inter}:
\begin{multline*}
I\lesssim \Big\|\Big(\int_1^te^{-s}S(-s)\int {\widehat{S(s)f}(\xi-\eta,s)\widehat{\partial_{x}v_*}(\eta,s)}d\eta ds\Big)\Big\|_{L^2}\\
\lesssim \int_0^t e^{-s}\Big\|\mathcal{F}\Big(u(s)\partial_{x}v_*\Big)\Big\|_{L^2}ds \lesssim \|\partial_x v_*\|_{L^2}\int_0^te^{-s}\|u(s)\|_{L^{\infty}}ds\\ \lesssim
\|\partial_x v_*\|_{L^2}\|u\|_{Y_T}\int_0^te^{-s}s^{\frac 12}ds \lesssim \|\partial_x v_*\|_{L^2}\|u\|_{Y_T}.
\end{multline*}
Similarly,
$$II\lesssim  \int_0^t e^{-s}\Big\|\mathcal{F}\Big(xf(s)v_*\Big)\Big\|_{L^2}ds\lesssim \|v_*\|_{L^2}\|u\|_{Y_T}\int_1^te^{-s}s^{\alpha}ds\lesssim \|v_*\|_{L^2}\|u\|_{Y_T}.$$
We now estimate $III$. First, notice that
\begin{multline*}
\partial_{\xi}F(s,\xi,\xi-\sigma,\eta)=\\
2i\int_1^s(s'-s)\eta e^{i\frac 12(\eta^2-2\eta(\xi-\sigma))(s-s')}e^{-(s-s')}\hat{f}(\xi-\eta,s)\overline{\hat{f}(\xi-\eta-\sigma,s')}\hat{f}(\xi-\sigma,s')ds'\\
+\int_1^se^{i\frac 12(\eta^2-2\eta(\xi-\sigma))(s-s')}e^{-(s-s')}\partial_{\xi}(\hat{f}(\xi-\eta,s)\overline{\hat{f}(\xi-\eta-\sigma,s')}\hat{f}(\xi-\sigma,s'))ds'.
\end{multline*}
Writing $\eta=(\xi-\sigma)-(\xi-\sigma-\eta)$,
\begin{multline*}
\partial_{\xi}F(s,\xi,\xi-\sigma,\eta)=\\
-2\int_1^s(s'-s) e^{i\frac 12(\eta^2-2\eta(\xi-\sigma))(s-s')}e^{-(s-s')}\hat{f}(\xi-\eta,s)\overline{\hat{f}(\xi-\eta-\sigma,s')}\widehat{\partial_xf}(\xi-\sigma,s')ds'\\
-2\int_1^s(s'-s) e^{i\frac 12(\eta^2-2\eta(\xi-\sigma))(s-s')}e^{-(s-s')}\hat{f}(\xi-\eta,s)\overline{\widehat{\partial_x f}(\xi-\eta-\sigma,s')}\widehat{f}(\xi-\sigma,s')ds'\\
+\int_1^se^{i\frac 12(\eta^2-2\eta(\xi-\sigma))(s-s')}e^{-(s-s')}\partial_{\xi}(\hat{f}(\xi-\eta,s)\overline{\hat{f}(\xi-\eta-\sigma,s')}\hat{f}(\xi-\sigma,s'))ds'\\
:=G_1(s,\xi,\eta,\sigma)+G_2(s,\xi,\eta,\sigma)+G_3(s,\xi,\eta,\sigma).
\end{multline*}
Now, 
\begin{multline*}
\Big\|\int_1^t\iint e^{-i\eta\sigma s} G_1(s,\xi,\eta,\sigma)d\eta d\sigma ds\Big\|_{L^2}
\lesssim \Big\|\int_1^t\int_1^s(s'-s)e^{-(s-s')}\\
\iint e^{i\frac 12(\eta^2-2\eta(\xi-\sigma))(s-s')}e^{-i\eta\sigma s} \hat{f}(\xi-\eta,s)\overline{\widehat{\partial_x f}(\xi-\eta-\sigma,s')}\widehat{\partial_xf}(\xi-\sigma,s')d\eta d\sigma ds'ds\Big\|_{L^2}.
\end{multline*}
By redistributing the phase term, that is, observing that
$$(\eta^2-2\eta(\xi-\sigma))(s-s')-2\sigma \eta s=-\xi^2s+(\xi-\eta)^2s-(\xi-\eta-\sigma)^2s'+(\xi-\sigma)^2s',$$
we
obtain
\begin{multline*}
\Big\|\int_1^t\iint e^{-2i\eta\sigma s} G_1(s,\xi,\eta,\sigma)d\eta d\sigma ds\Big\|_{L^2}\\
\lesssim \Big\|\mathcal{F}\int_1^t\int_1^s(s'-s)e^{-(s-s')}S(-s)\Big[S(s)f(s)\overline{S(s')f(s')}S(s')\partial_xf(s')]ds'ds\Big\|_{L^2}\\
\lesssim  \int_1^t\|u(s)\|_{L^{\infty}}\int_1^s(s-s')e^{-(s-s')}\|u(s')\|_{L^{\infty}}\|\partial_xu(s')\|_{L^{2}}ds'ds\\
\lesssim \|u\|_{Y_T}^3\int_1^ts^{-\frac 12}\int_1^se^{-\frac 12(s-s')}{s'}^{-\frac 12+\alpha}ds'ds \lesssim \|u\|_{Y_T}^3\int_1^t s^{\alpha-1}ds\lesssim t^{\alpha}\|u\|_{Y_T}^3
\end{multline*}
by \eqref{estfant}, and, analogously,
$$\Big\|\int_1^t\iint e^{-i\eta\sigma s} G_2(s,\xi,\eta,\sigma)d\eta d\sigma ds\Big\|_{L^2}\lesssim t^{\alpha}\|u\|_{Y_T}^3.$$
With respect to the third integral, similar computations yield
\begin{multline*} 
\Big\|\int_1^t\iint e^{-i\eta\sigma s} G_3(s,\xi,\eta,\sigma)d\eta d\sigma ds\Big\|_{L^2}\\\lesssim \int_1^t\int_1^se^{-(s-s')}\|xf(s)\|_{L^2}\|u(s')\|_{L^{\infty}}^2ds'ds\\
+\int_1^t\int_1^se^{-(s-s')}\|xf(s')\|_{L^2}\|u(s)\|_{L^{\infty}}\|u(s')\|_{L^{\infty}}ds'ds\\
\lesssim \|u\|_{Y_T}^3\Big(\int_1^t\int_1^s e^{-(s-s')}{s'}^{-1}s^{\alpha}ds'ds+\int_1^t\int_1^s e^{-(s-s')}{s'}^{\alpha-\frac 12}s^{-\frac 12}ds'ds\Big)\lesssim \|u\|_{Y_T}^3t^{\alpha}.
\end{multline*}
Finally, $III\lesssim \|u\|_{Y_T}^3t^{\alpha}$ and we obtain
\begin{displaymath}
\|xf\|_{L^2}\lesssim \|xu_*\|_{L^2}+\|v_*\|_{H^1}\|u\|_{Y_T}+\|u\|_{Y_T}^3t^{\alpha},
\end{displaymath}
that is
\begin{equation}
\label{estimativac}
\|t^{-\alpha}xf\|_{L_T^{\infty}L^2_x}\lesssim \|\partial_xu_*\|_{L^2}+\|v_*\|_{H^1}\|u\|_{Y_T}+\|u\|_{Y_T}^3.
\end{equation}
{\bf Estimate for $\|t^{\frac 12}u\|_{L_t^{\infty}L^{\infty}_x}$}: From \eqref{fchapeu3}, we write
\begin{displaymath}\label{eq:fchapeu}
\begin{array}{lllll}
\hat{f}(\xi)&=&\displaystyle \hat{u_*}+i\int_1^t\int e^{-i\frac 12(\xi^2-(\xi-\eta)^2)s} e^{-s}{\hat{f}(\xi-\eta,s)\hat{v_*}(\eta,s)}d\eta ds\\
&&\displaystyle +i\int_1^t\int_1^s\iint e^{i\phi}e^{-(s-s')}\hat{f}(\sigma,s')\overline{\hat{f}(\sigma-\eta,
	s')} \hat{f}(\xi-\eta,s)d\eta d\sigma ds'ds\\
\\
&=&\displaystyle \hat{u_*}+i\int_1^tR_1(\xi,s)ds+\displaystyle i\int_1^t\int_1^s\iint e^{i\phi} G(s',s,\xi,\eta,\sigma)d\eta d\sigma ds'ds,$$
\end{array}
\end{displaymath}
where
$$G(s',s,\xi,\eta,\sigma)=e^{-(s-s')}\hat{f}(\sigma,s')\overline{\hat{f}(\sigma-\eta,
	s')} \hat{f}(\xi-\eta,s),$$
$$R_1(\xi,s)=\int e^{-i\frac 12(\xi^2-(\xi-\eta)^2)s} e^{-s}{\hat{f}(\xi-\eta,s)\hat{v_*}(\eta,s)}d\eta$$
and
$$\phi=-\xi\eta s+\frac 12\eta^2s-\frac 12\eta^2s'+\eta\sigma s'=-\xi\eta s+\eta \tilde{\sigma}s,$$
with $\tilde{\sigma}=\frac 1{2s}(\eta(s-s')+2\sigma s')$. By performing the change of variables $\sigma \mapsto \tilde{\sigma}$, we obtain
\begin{multline}
\label{interbis}
\int_1^t\int_1^s \iint e^{i\phi}G(s,\xi,\eta,\sigma)d\eta d\sigma ds'ds\\
=\int_1^t\int_1^s\iint e^{i\eta\xi s}e^{-i\eta\sigma s}\frac s{s'} G\Big(s',s,\xi,\eta,\sigma\frac{s}{s'}-\eta\frac{s-s'}{2s'}\Big)d\eta d\sigma ds' ds\\
=\int_1^t\int_1^s\iint \frac s{s'}\mathcal{F}_{\eta,\sigma}(e^{i\eta\xi s}e^{-i\eta\sigma s}) \mathcal{F}_{\eta,\sigma}^{-1}\Big(G\Big(s',s,\xi,\eta,\sigma\frac{s}{s'}-\eta\frac{s-s'}{2s'}\Big)\Big)d\eta' d\sigma' ds' ds
\end{multline}
by Plancherel. Furthermore,
$$\mathcal{F}_{\eta,\sigma}(e^{i\eta\xi s}e^{-i\eta\sigma s})=\mathcal{F}_{\eta,\sigma}(e^{-i\eta\sigma s})(\eta'-\xi s,\sigma')=\frac 1{2s}e^{i\frac 12\frac{i\eta'\sigma'}{2s}}e^{-i\frac 12\xi\sigma'}.$$
On the other hand,
\begin{align*}
&\mathcal{F}_{\sigma}^{-1}\Big(G\Big(s',s,\xi,\eta,\sigma\frac{s}{s'}-\eta\frac{s-s'}{2s'}\Big)\Big)\\
=\ &e^{-(s-s')}\hat{f}(\xi-\eta,s)\mathcal{F}^{-1}_{\sigma}\Big(\overline{\hat{f}\Big(\frac{s}{s'}\Big(\sigma-\eta\frac{s+s'}{4s}\Big),s')}\hat{f}\Big(\frac{s}{s'}\Big(\sigma-\eta\frac{s+s'}{2s}+\eta\frac{s'}s\Big)\Big),s'\Big)\Big)\Big)\\
=\ &e^{-(s-s')}\hat{f}(\xi-\eta,s)e^{i\sigma'\eta\frac{s+s'}{4s}}\mathcal{F}^{-1}_{\sigma}\Big(\overline{\hat{f}\Big(\frac s{s'}\sigma,s'\Big)}\hat{f}\Big(\frac s{s'}\sigma+\eta,s'\Big)\Big)\Big)\\=\ &\frac {s'}{s}\hat{f}(\xi-\eta,s)e^{i\sigma'\eta\frac{s+s'}{4s}}\mathcal{F}^{-1}_{\sigma}\Big(\overline{\hat{f}}(\sigma,s')\hat{f}(\sigma+\eta,s')\Big)\Big(\frac{s'\sigma'}s,\eta'\Big)\\
=\ &\frac {s'}{s}\hat{f}(\xi-\eta,s)e^{i\sigma'\eta\frac{s+s'}{4s}}[\check{\overline{f}}*_{\sigma}e^{-i\sigma\eta}f]\Big(\frac{s'\sigma'}{s}\Big)\\=\ &\frac{s'}{s}\hat{f}(\xi-\eta,s)e^{i\sigma'\eta\frac{s+s'}{4s}}\int e^{-i\eta x}f(x,s')\overline{f}\Big(x-\frac{s'\sigma'}s,s'\Big)dx,
\end{align*}
where $\check{g}(x)=g(-x)$. Also,
\begin{multline*}
\mathcal{F}_{\eta}^{-1}[\hat{f}(\xi-\eta,s)e^{i\frac 12\sigma'\eta\frac{s+s'}{2s}} e^{-i\eta x}](\eta')=
\mathcal{F}_{\eta}^{-1}[\hat{f}(\xi-\eta,s)](\eta'+\sigma'\frac{s+s'}{2s}-x)\\
=[e^{i\eta\xi}\check{f}]\Big(\eta'+\sigma'\frac{s+s'}{2s}-x,s\Big)=e^{\frac 12i(\eta'+\sigma'\frac{s+s'}{2s}-x)\xi}f\Big(x-\sigma'\frac{s+s'}{2s}-\eta',s\Big)
\end{multline*}
Finally,
\begin{multline*}
\mathcal{F}_{\eta,\sigma}^{-1}\Big(G\Big(s',s,\xi,\eta,\sigma\frac{s}{s'}-\eta\frac{s-s'}{2s'}\Big)\Big)\\=e^{-(s-s')}\frac {s'}s\int e^{i\frac 12(\eta'+\sigma'\frac{s+s'}{2s}-x)\xi}f\Big(x-\sigma'\frac{s+s'}{2s}-\eta',s\Big)f(x,s')\overline{f}\Big(x-\frac{s'\sigma'}s,s'\Big)dx.
\end{multline*}
By \eqref{interbis},
\begin{multline*}
\int_1^t\int_1^s \iint e^{i\phi}G(s,\xi,\eta,\sigma)d\eta d\sigma ds'ds\\
=\int_1^t\int_1^s\iint \frac 1{2s'}e^{i\frac{i\eta'\sigma'}{4s}}e^{-\frac 12i\xi\sigma'} \mathcal{F}_{\eta,\sigma}^{-1}\Big(G\Big(s',s,\xi,\eta,\sigma\frac{s}{s'}-\eta\frac{s-s'}{2s'}\Big)\Big)d\eta' d\sigma' ds' ds\\
=\int_1^t\int_1^s\iint \frac 1{2s'}e^{-\frac 12i\xi\sigma'} \mathcal{F}_{\eta,\sigma}^{-1}\Big(G\Big(s',s,\xi,\eta,\sigma\frac{s}{s'}-\eta\frac{s-s'}{2s'}\Big)\Big) d\eta' d\sigma' ds' ds\\
\\
+\int_1^t\int_1^s\iint \frac 1{2s'}(e^{i\frac{i\eta'\sigma'}{2s}}-1)e^{-\frac 12i\xi\sigma'} \mathcal{F}_{\eta,\sigma}^{-1}\Big(G\Big(s',s,\xi,\eta,\sigma\frac{s}{s'}-\eta\frac{s-s'}{2s'}\Big)\Big) d\eta' d\sigma' ds' ds\\
=\int_1^t\int_1^s\frac 1{2s'}G(s',s,\xi,0,\xi\frac s{s'})ds'ds+\int_1^tR_2(\xi,s)ds\\
=\int_1^t\Big(\int_1^s\frac 1{2s'}e^{-(s-s')}\Big|\hat{f}\Big(\frac s{s'}\xi,s'\Big)\Big|^2ds'\Big)\hat{f}(\xi,s)ds+\int_1^tR_2(\xi,s)ds
\end{multline*}
where
\begin{equation}
R_2(\xi,s)=\int_1^s
\iint \frac 1{2s'}(e^{i\frac{i\eta'\sigma'}{4s}}-1)e^{-\frac 12i\xi\sigma'} 
\mathcal{F}_{\eta,\sigma}^{-1}\Big(G\Big(s',s,\xi,\eta,\sigma\frac{s}{s'}-\eta\frac{s-s'}{2s'}\Big)\Big) d\eta' d\sigma' ds'
\end{equation}
Furthermore,

$$|R_1(\xi,s)|\leq \|u\|_{Y_T}\|v_*\|_{L^2}e^{-s}$$
and
\begin{multline*}
|R_2(\xi,s)|
\lesssim \\\int_1^s\frac 1{s}e^{-(s-s')}\Big|\sin\Big(\frac{\eta'\sigma'}{4s'}\Big)\Big|\Big|f\Big(x-\sigma'\frac{s+s'}{2s}-\eta',s\Big)\Big||f(x,s')|\Big|f\Big(x-\frac{s'\sigma'}s,s'\Big)\Big|ds'd\eta'd\sigma'dx\\\lesssim
\int_1^s\frac 1{s{s'}^{\delta}}e^{-(s-s')}|\eta'|^{\delta}|\sigma'|^{\delta}\Big|f\Big(x-\sigma'\frac{s+s'}{2s}-\eta',s\Big)\Big||f(x,s')|\Big|f\Big(x-\frac{s'\sigma'}s,s'\Big)\Big|ds'd\eta'd\sigma'dx
\end{multline*}
for all $0<\delta<\frac 12$. Also, observe that
\begin{multline*}
|\eta'|^{\delta}\leq \Big|x-\sigma'\frac{s+s'}{2s}-\eta'\Big|^{\delta}+\Big|x-\sigma'\frac{s+s'}{2s}\Big|^{\delta}\leq
\Big|x-\sigma'\frac{s+s'}{2s}-\eta'\Big|^{\delta}\\
+\Big(\frac{s+s'}{2s'}\Big)^{\delta}\Big|x-\frac {s'}s\sigma'\Big|^{\delta}+\Big(1+\frac{s+s'}{2s'}\Big)^{\delta}|x|^{\delta}
\end{multline*}
and 
$$|\sigma'|^{\delta}\leq \Big(\frac s{s'}\Big)^{\delta}\Big|x-\frac {s'}s\sigma'\Big|^{\delta}+\Big(\frac s{s'}\Big)^{\delta}|x|^{\delta}.$$
We now have six terms to estimate. We will only treat one of the most significant, the remaining follow by similar computations:
\begin{multline*}
\int_1^s\frac 1{s^{1-2\delta}{s'}^{3\delta}}e^{-(s-s')}|x|^{2\delta}\Big|f\Big(x-\sigma'\frac{s+s'}{2s}-\eta',s\Big)\Big||f(x,s')|\Big|f\Big(x-\frac{s'\sigma'}s,s'\Big)\Big|ds'dxd\eta'd\sigma'\\
\lesssim \|u\|_{Y_T}^3\int_1^se^{-(s-s')}s'^{2\alpha-3\delta-1}s^{2\delta+\alpha}ds'\lesssim \|u\|_{Y_T}^3 s^{3\alpha-\delta-1}.
\end{multline*}
Hence, taking the time derivative of \eqref{eq:fchapeu},
\begin{equation}\label{eq:derivadafchapeu}
\partial_t\hat{f}(\xi,t)=i\Big(\int_1^t\frac 1{2s'}e^{-(t-s')}\Big|\hat{f}\Big(\frac s{s'}\xi,s'\Big)\Big|^2ds'\Big)\hat{f}(\xi,t)+\int_1^tR(\xi,s),
\end{equation}
$$|R(\xi,s)|\leq |R_1(\xi,s)+R_2(\xi,s)|\lesssim \|v_*\|_{L^2}\|u\|_{Y_T}e^{-s}+\|u\|_{Y_T}^3s^{3\alpha-\delta-1}.$$

\noindent We can now complete the proof of the $L^{\infty}$ estimate: fix $\alpha,\beta,\delta>0$ so that $3\alpha<\delta<1/2$ and $\alpha<\beta<1/4$ and write
$$\Psi(\xi,t)=\int_1^t\int_1^s\frac 1{2s'}e^{-(s-s')}\Big|\hat{f}\Big(\frac s{s'}\xi,s'\Big)\Big|^2ds'ds\quad\textrm{ and }\quad \hat{w}(\xi,t)=\hat{f}(\xi,t)e^{i\Psi(\xi,t)}.$$
Using the integrating factor $e^{i\Psi(\xi,t)}$ in \eqref{eq:derivadafchapeu},
$$
\hat{w}_t = e^{i\Psi(\xi,t)}R(t\xi), \quad t>0.
$$
Therefore
$$|\hat{f}(\xi,t)|=|\hat{w}(\xi,t)|\leq |\hat{w}(\xi,0)|+\int_1^t |\partial_t\hat{w}(\xi,s)|ds\leq |\hat{f}(\xi,0)|+\int_1^t|R(\xi,s)|ds$$
$$\lesssim \|xu_*\|_{L^2}+\|v_*\|_{L^2}\|u\|_{Y_T}e^{-t}+\|u\|_{Y_T}^3t^{3\alpha-\delta}.$$
Using Lemma 2.2 in \cite{181},
$$\|u(t)\|_{L^{\infty}}\leq \frac 1{t^{\frac 12}}\|\hat{f}\|_{L^{\infty}}+\frac 1{t^{\frac 12+\beta}}\|xf\|_{L^2}\lesssim \frac 1{t^{\frac 12}}\Big(\|xu_*\|_{L^2}+\|v_*\|_{L^2}\|u\|_{Y_T}+\|u\|_{Y_T}^3)$$
and
\begin{equation}
\label{estimativaa}
\|t^{\frac 12}u\|_{L_t^{\infty}L^{\infty}_x}\lesssim \|xu_*\|_{L^2}+\|v_*\|_{L^2}\|u\|_{Y_T}+\|u\|_{Y_T}^3,
\end{equation}
which concludes the proof of Proposition \ref{modifiedapriori}.
\end{proof}
\begin{proof}[Proof of Theorem \ref{modified}]
From Proposition \ref{modifiedapriori} it is straightforward, once again by an obstruction argument, to show  that $\|u\|_{Y_T}$ remains uniformly bounded as $t\to+\infty$ for small initial data. The decay rate of $\|u\|_{L^{\infty}}$ and the scattering result are then immediate consequences of the proof of the $L^{\infty}$ estimate \eqref{estimativaa}.
\end{proof}

\medskip

\noindent
\textbf{Acknowledgements} Sim\~ao Correia was partially supported by Funda\c{c}\~ao para a Ci\^encia e Tecnologia, through the grant SFRH/BD/96399/2013 and through contract
UID/MAT/04561/2013.\\ Filipe Oliveira  was partially supported by the Project CEMAPRE - UID/ MULTI/00491/2013 financed by FCT/MCTES through national funds. 

\bibliography{BibliotecaCO}
\bibliographystyle{plain}
\small\noindent\textsc{Sim\~ao Correia}\\
CMAF-CIO and FCUL\\
Universidade de Lisboa\\
\noindent	
Campo Grande, Edif\'\i cio C6, Piso 2, 1749-016 Lisboa, Portugal\\
\verb"sfcorreia@fc.ul.pt"\\

\small\noindent\textsc{Filipe Oliveira}\\
Mathematics Department and CEMAPRE\\
\noindent	
	ISEG, Universidade de Lisboa\\
	Rua do Quelhas 6, 1200-781 Lisboa, Portugal\\
\verb"foliveira@iseg.ulisboa.pt"\\

\end{document}